\documentclass[12pt,oneside,english,american]{amsart}
\usepackage[T1]{fontenc}
\usepackage[utf8]{inputenc}
\usepackage{amsbsy}
\usepackage{amstext}
\usepackage{amsthm}
\usepackage{amssymb}
\usepackage{stmaryrd}
\usepackage[a4paper]{geometry}
\usepackage[all]{xy}

\makeatletter
\numberwithin{equation}{section}
\numberwithin{figure}{section}
\theoremstyle{plain}
\newtheorem{thm}{\protect\theoremname}[section]
\theoremstyle{definition}
\newtheorem{example}[thm]{\protect\examplename}
\theoremstyle{remark}
\newtheorem{rem}[thm]{\protect\remarkname}
\theoremstyle{definition}
\newtheorem{defn}[thm]{\protect\definitionname}
\theoremstyle{plain}
\newtheorem{prop}[thm]{\protect\propositionname}
\theoremstyle{plain}
\newtheorem{lem}[thm]{\protect\lemmaname}
\theoremstyle{plain}
\newtheorem{cor}[thm]{\protect\corollaryname}
\theoremstyle{remark}
\newtheorem{claim}[thm]{\protect\claimname}

\usepackage[hidelinks]{hyperref}

\makeatother

\usepackage{babel}
\addto\captionsamerican{\renewcommand{\claimname}{Claim}}
\addto\captionsamerican{\renewcommand{\corollaryname}{Corollary}}
\addto\captionsamerican{\renewcommand{\definitionname}{Definition}}
\addto\captionsamerican{\renewcommand{\examplename}{Example}}
\addto\captionsamerican{\renewcommand{\lemmaname}{Lemma}}
\addto\captionsamerican{\renewcommand{\propositionname}{Proposition}}
\addto\captionsamerican{\renewcommand{\remarkname}{Remark}}
\addto\captionsamerican{\renewcommand{\theoremname}{Theorem}}
\addto\captionsenglish{\renewcommand{\claimname}{Claim}}
\addto\captionsenglish{\renewcommand{\corollaryname}{Corollary}}
\addto\captionsenglish{\renewcommand{\definitionname}{Definition}}
\addto\captionsenglish{\renewcommand{\examplename}{Example}}
\addto\captionsenglish{\renewcommand{\lemmaname}{Lemma}}
\addto\captionsenglish{\renewcommand{\propositionname}{Proposition}}
\addto\captionsenglish{\renewcommand{\remarkname}{Remark}}
\addto\captionsenglish{\renewcommand{\theoremname}{Theorem}}
\providecommand{\claimname}{Claim}
\providecommand{\corollaryname}{Corollary}
\providecommand{\definitionname}{Definition}
\providecommand{\examplename}{Example}
\providecommand{\lemmaname}{Lemma}
\providecommand{\propositionname}{Proposition}
\providecommand{\remarkname}{Remark}
\providecommand{\theoremname}{Theorem}

\begin{document}
\title{Orbifold pseudo-effective cones of toric stacks}
\author{Ratko Darda and Takehiko Yasuda}
\address{Faculty of Engineering and Natural Sciences, Sabancı University, Tuzla,
Istanbul, Turkey}
\email{ratko.darda@gmail.com, ratko.darda@sabanciuniv.edu}
\address{Department of Mathematics, Graduate School of Science, the University
of Osaka, Toyonaka, Osaka 560-0043, JAPAN}
\email{yasuda.takehiko.sci@osaka-u.ac.jp}
\begin{abstract}
In this paper, we explicitly describe the orbifold pseudo-effective
cone of a split toric stack.
\end{abstract}

\maketitle
\selectlanguage{english}%
\global\long\def\bigmid{\mathrel{}\middle|\mathrel{}}%

\global\long\def\AA{\mathbb{A}}%

\global\long\def\CC{\mathbb{C}}%

\global\long\def\EE{\mathbb{E}}%

\global\long\def\FF{\mathbb{F}}%

\global\long\def\GG{\mathbb{G}}%

\global\long\def\LL{\mathbb{L}}%

\global\long\def\MM{\mathbb{M}}%

\global\long\def\NN{\mathbb{N}}%

\global\long\def\PP{\mathbb{P}}%

\global\long\def\QQ{\mathbb{Q}}%

\global\long\def\RR{\mathbb{R}}%

\global\long\def\SS{\mathbb{S}}%

\global\long\def\ZZ{\mathbb{Z}}%

\global\long\def\bA{\mathbf{A}}%

\global\long\def\ba{\mathbf{a}}%

\global\long\def\bb{\mathbf{b}}%

\global\long\def\bc{\mathbf{c}}%

\global\long\def\bd{\mathbf{d}}%

\global\long\def\bf{\mathbf{f}}%

\global\long\def\bg{\mathbf{g}}%

\global\long\def\bh{\mathbf{h}}%

\global\long\def\bj{\mathbf{j}}%

\global\long\def\bk{\mathbf{k}}%

\global\long\def\bm{\mathbf{m}}%

\global\long\def\bp{\mathbf{p}}%

\global\long\def\bq{\mathbf{q}}%

\global\long\def\br{\mathbf{r}}%

\global\long\def\bs{\mathbf{s}}%

\global\long\def\bt{\mathbf{t}}%

\global\long\def\bu{\mathbf{u}}%

\global\long\def\bv{\mathbf{v}}%

\global\long\def\bw{\mathbf{w}}%

\global\long\def\bx{\boldsymbol{x}}%

\global\long\def\by{\boldsymbol{y}}%

\global\long\def\bz{\mathbf{z}}%

\global\long\def\cA{\mathcal{A}}%

\global\long\def\cB{\mathcal{B}}%

\global\long\def\cC{\mathcal{C}}%

\global\long\def\cD{\mathcal{D}}%

\global\long\def\cE{\mathcal{E}}%

\global\long\def\cF{\mathcal{F}}%

\global\long\def\cG{\mathcal{G}}%

\global\long\def\cH{\mathcal{H}}%

\global\long\def\cI{\mathcal{I}}%

\global\long\def\cJ{\mathcal{J}}%

\global\long\def\cK{\mathcal{K}}%

\global\long\def\cL{\mathcal{L}}%

\global\long\def\cM{\mathcal{M}}%

\global\long\def\cN{\mathcal{N}}%

\global\long\def\cO{\mathcal{O}}%

\global\long\def\cP{\mathcal{P}}%

\global\long\def\cQ{\mathcal{Q}}%

\global\long\def\cR{\mathcal{R}}%

\global\long\def\cS{\mathcal{S}}%

\global\long\def\cT{\mathcal{T}}%

\global\long\def\cU{\mathcal{U}}%

\global\long\def\cV{\mathcal{V}}%

\global\long\def\cW{\mathcal{W}}%

\global\long\def\cX{\mathcal{X}}%

\global\long\def\cY{\mathcal{Y}}%

\global\long\def\cZ{\mathcal{Z}}%

\global\long\def\fa{\mathfrak{a}}%

\global\long\def\fb{\mathfrak{b}}%

\global\long\def\fc{\mathfrak{c}}%

\global\long\def\fd{\mathfrak{d}}%

\global\long\def\fe{\mathfrak{e}}%

\global\long\def\ff{\mathfrak{f}}%

\global\long\def\fj{\mathfrak{j}}%

\global\long\def\fk{\mathfrak{k}}%

\global\long\def\fn{\mathfrak{n}}%

\global\long\def\fm{\mathfrak{m}}%

\global\long\def\fp{\mathfrak{p}}%

\global\long\def\fs{\mathfrak{s}}%

\global\long\def\ft{\mathfrak{t}}%

\global\long\def\fx{\mathfrak{x}}%

\global\long\def\fv{\mathfrak{v}}%

\global\long\def\fC{\mathfrak{C}}%

\global\long\def\fD{\mathfrak{D}}%

\global\long\def\fF{\mathfrak{F}}%

\global\long\def\fJ{\mathfrak{J}}%

\global\long\def\fG{\mathfrak{G}}%

\global\long\def\fK{\mathfrak{K}}%

\global\long\def\fM{\mathfrak{M}}%

\global\long\def\fN{\mathfrak{N}}%

\global\long\def\fO{\mathfrak{O}}%

\global\long\def\fS{\mathfrak{S}}%

\global\long\def\fV{\mathfrak{V}}%

\global\long\def\fX{\mathfrak{X}}%

\global\long\def\fY{\mathfrak{Y}}%

\global\long\def\ru{\mathrm{u}}%

\global\long\def\rv{\mathbf{\mathrm{v}}}%

\global\long\def\rw{\mathrm{w}}%

\global\long\def\rx{\mathrm{x}}%

\global\long\def\ry{\mathrm{y}}%

\global\long\def\rz{\mathrm{z}}%

\global\long\def\AdGp{\mathrm{AdGp}}%

\global\long\def\Aff{\mathbf{Aff}}%

\global\long\def\Alg{\mathbf{Alg}}%

\global\long\def\age{\operatorname{age}}%

\global\long\def\Ann{\mathrm{Ann}}%

\global\long\def\Aut{\operatorname{Aut}}%

\global\long\def\B{\operatorname{\mathrm{B}}}%

\global\long\def\Bl{\mathrm{Bl}}%

\global\long\def\Box{\mathrm{Box}}%

\global\long\def\bSigma{\mathbf{\Sigma}}%

\global\long\def\C{\operatorname{\mathrm{C}}}%

\global\long\def\calm{\mathrm{calm}}%

\global\long\def\can{\mathrm{can}}%

\global\long\def\center{\mathrm{center}}%

\global\long\def\characteristic{\operatorname{char}}%

\global\long\def\cjun{c\textrm{-jun}}%

\global\long\def\codim{\operatorname{codim}}%

\global\long\def\Coker{\mathrm{Coker}}%

\global\long\def\Conj{\operatorname{Conj}}%

\global\long\def\cw{\mathbf{cw}}%

\global\long\def\D{\mathrm{D}}%

\global\long\def\Df{\mathrm{Df}}%

\global\long\def\dec{\mathrm{dec}}%

\global\long\def\det{\operatorname{det}}%

\global\long\def\diag{\mathrm{diag}}%

\global\long\def\discrep#1{\mathrm{discrep}\left(#1\right)}%

\global\long\def\dom{\mathrm{dom}}%

\global\long\def\doubleslash{\sslash}%

\global\long\def\E{\operatorname{E}}%

\global\long\def\Emb{\operatorname{Emb}}%

\global\long\def\et{\textrm{ét}}%

\global\long\def\etop{\mathrm{e}_{\mathrm{top}}}%

\global\long\def\el{\mathrm{e}_{l}}%

\global\long\def\Exc{\mathrm{Exc}}%

\global\long\def\Ext{\operatorname{Ext}}%

\global\long\def\Fal{\mathrm{Fal}}%

\global\long\def\FConj{F\textrm{-}\Conj}%

\global\long\def\Fitt{\operatorname{Fitt}}%

\global\long\def\fMov{\overline{\mathfrak{Mov}}}%

\global\long\def\fPEff{\overline{\mathfrak{Eff}}}%

\global\long\def\fr{\mathrm{fr}}%

\global\long\def\Fr{\mathrm{Fr}}%

\global\long\def\Gal{\operatorname{Gal}}%

\global\long\def\GalGps{\mathrm{GalGps}}%

\global\long\def\GL{\mathrm{GL}}%

\global\long\def\Gor{\mathrm{Gor}}%

\global\long\def\Grass{\mathrm{Grass}}%

\global\long\def\gw{\mathbf{gw}}%

\global\long\def\H{\operatorname{\mathrm{H}}}%

\global\long\def\hattimes{\hat{\times}}%

\global\long\def\hatotimes{\hat{\otimes}}%

\global\long\def\Hilb{\mathrm{Hilb}}%

\global\long\def\Hodge{\mathrm{Hodge}}%

\global\long\def\Hom{\operatorname{Hom}}%

\global\long\def\hyphen{\textrm{-}}%

\global\long\def\I{\operatorname{\mathrm{I}}}%

\global\long\def\id{\mathrm{id}}%

\global\long\def\Image{\operatorname{\mathrm{Im}}}%

\global\long\def\ind{\mathrm{ind}}%

\global\long\def\injlim{\varinjlim}%

\global\long\def\Inn{\mathrm{Inn}}%

\global\long\def\iper{\mathrm{iper}}%

\global\long\def\Iso{\operatorname{Iso}}%

\global\long\def\isoto{\xrightarrow{\sim}}%

\global\long\def\J{\operatorname{\mathrm{J}}}%

\global\long\def\Jac{\mathrm{Jac}}%

\global\long\def\kConj{k\textrm{-}\Conj}%

\global\long\def\KConj{K\textrm{-}\Conj}%

\global\long\def\Ker{\operatorname{Ker}}%

\global\long\def\Kzero{\operatorname{K_{0}}}%

\global\long\def\lc{\mathrm{lc}}%

\global\long\def\lcr{\mathrm{lcr}}%

\global\long\def\lcm{\operatorname{\mathrm{lcm}}}%

\global\long\def\length{\operatorname{\mathrm{length}}}%

\global\long\def\M{\operatorname{\mathrm{M}}}%

\global\long\def\MC{\mathrm{MC}}%

\global\long\def\MHS{\mathbf{MHS}}%

\global\long\def\mld{\mathrm{mld}}%

\global\long\def\mod#1{\pmod{#1}}%

\global\long\def\Mov{\overline{\mathrm{Mov}}}%

\global\long\def\mRep{\mathbf{mRep}}%

\global\long\def\mult{\mathrm{mult}}%

\global\long\def\N{\operatorname{\mathrm{N}}}%

\global\long\def\Nef{\mathrm{Nef}}%

\global\long\def\nor{\mathrm{nor}}%

\global\long\def\nr{\mathrm{nr}}%

\global\long\def\NS{\mathrm{NS}}%

\global\long\def\op{\mathrm{op}}%

\global\long\def\orb{\mathrm{orb}}%

\global\long\def\ord{\operatorname{ord}}%

\global\long\def\P{\operatorname{P}}%

\global\long\def\PEff{\overline{\mathrm{Eff}}}%

\global\long\def\PGL{\mathrm{PGL}}%

\global\long\def\pt{\mathbf{pt}}%

\global\long\def\pur{\mathrm{pur}}%

\global\long\def\perf{\mathrm{perf}}%

\global\long\def\perm{\mathrm{perm}}%

\global\long\def\Pic{\mathrm{Pic}}%

\global\long\def\pr{\mathrm{pr}}%

\global\long\def\Proj{\operatorname{Proj}}%

\global\long\def\projlim{\varprojlim}%

\global\long\def\Qbar{\overline{\QQ}}%

\global\long\def\QConj{\mathbb{Q}\textrm{-}\Conj}%

\global\long\def\R{\operatorname{\mathrm{R}}}%

\global\long\def\Ram{\operatorname{\mathrm{Ram}}}%

\global\long\def\rank{\operatorname{\mathrm{rank}}}%

\global\long\def\rat{\mathrm{rat}}%

\global\long\def\Ref{\mathrm{Ref}}%

\global\long\def\rig{\mathrm{rig}}%

\global\long\def\rj{\mathrm{rj}}%

\global\long\def\red{\mathrm{red}}%

\global\long\def\reg{\mathrm{reg}}%

\global\long\def\rep{\mathrm{rep}}%

\global\long\def\Rep{\mathbf{Rep}}%

\global\long\def\sbrats{\llbracket s\rrbracket}%

\global\long\def\Sch{\mathbf{Sch}}%

\global\long\def\sep{\mathrm{sep}}%

\global\long\def\Set{\mathbf{Set}}%

\global\long\def\sing{\mathrm{sing}}%

\global\long\def\SL{\mathrm{SL}}%

\global\long\def\sm{\mathrm{sm}}%

\global\long\def\small{\mathrm{small}}%

\global\long\def\Sp{\operatorname{Sp}}%

\global\long\def\Spec{\operatorname{Spec}}%

\global\long\def\Spf{\operatorname{Spf}}%

\global\long\def\ss{\mathrm{ss}}%

\global\long\def\st{\mathrm{st}}%

\global\long\def\Stab{\operatorname{Stab}}%

\global\long\def\Supp{\operatorname{Supp}}%

\global\long\def\spars{\llparenthesis s\rrparenthesis}%

\global\long\def\Sym{\mathrm{Sym}}%

\global\long\def\T{\operatorname{T}}%

\global\long\def\tame{\mathrm{tame}}%

\global\long\def\tbrats{\llbracket t\rrbracket}%

\global\long\def\top{\mathrm{top}}%

\global\long\def\tors{\mathrm{tors}}%

\global\long\def\tpars{\llparenthesis t\rrparenthesis}%

\global\long\def\tor{\mathrm{tor}}%

\global\long\def\Tr{\mathrm{Tr}}%

\global\long\def\ulAut{\operatorname{\underline{Aut}}}%

\global\long\def\ulHom{\operatorname{\underline{Hom}}}%

\global\long\def\ulInn{\operatorname{\underline{Inn}}}%

\global\long\def\ulIso{\operatorname{\underline{{Iso}}}}%

\global\long\def\ulSpec{\operatorname{\underline{{Spec}}}}%

\global\long\def\Utg{\operatorname{Utg}}%

\global\long\def\utg{\mathrm{utg}}%

\global\long\def\Unt{\operatorname{Unt}}%

\global\long\def\Var{\mathbf{Var}}%

\global\long\def\Vol{\mathrm{Vol}}%

\global\long\def\wt{\mathrm{wt}}%

\global\long\def\Y{\operatorname{\mathrm{Y}}}%
\selectlanguage{american}%

\section{Introduction}

In \cite{darda2024thebatyrevtextendashmanin}, we introduced the notion
of the orbifold pseudo-effective cones in order to generalize the
Batyrev--Manin conjecture to Deligne--Mumford stacks. For varieties,
this notion reduces to the notion of usual pseudo-effective cones.
In \cite{darda2023themanin}, we verified the conjecture in the case
where the stack in question is a split toric stack and the height
function is the one associated to the anti-canonical raised line bundle.
In the same paper, we presented a candidate description of the orbifold
pseudo-effective cone of a split toric stack \cite[Remark 3.4.2]{darda2023themanin},
however we could not verify that the description was correct. The
aim of this paper is to prove that it is indeed the correct description
of the orbifold pseudo-effective cone. 

To state our main result more precisely, let us first recall how we
can describe the pseudo-effective cone of a complete smooth toric
variety. In what follows, we work over a field $k$ of characteristic
zero. Let $\Sigma$ be a complete smooth fan in $N_{\RR}=N\otimes\RR$,
where $N$ is a free abelian group of finite rank. Let $X$ be the
toric variety associated to $\Sigma$ over $k$ and let $E_{\rho}$
be the torus-invariant prime divisors on $X$ corresponding to one-dimensional
cones $\rho\in\Sigma(1)$ respectively, where $\Sigma(1)$ denotes
the set of one-dimensional cones in $\Sigma$. The pseudo-effective
cone $\PEff(X)$ of $X$, which is a cone in the Néron--Severi space
$\N^{1}(X)$, is the cone generated by the numerical classes $[E_{\rho}]$,
$\rho\in\Sigma(1)$ of torus-invariant prime divisors. 

Next, consider a stacky fan $\bSigma=(\Sigma,N,\beta)$ in the sense
of Borisov--Chen--Smith \cite{borisov2004theorbifold} and the associated
toric stack $\cX$ over $k$. We again assume that the fan $\Sigma$
is complete. So, $\cX$ is a proper and smooth Deligne--Mumford stack
over $k$. Note that $\cX$ is a so-called split toric stack, namely
the stacky torus contained in $\cX$ as an open substack is isomorphic
to $(\GG_{m})^{d}\times\prod_{i=1}^{s}\B\mu_{l_{i}}$ over $k$, where
$d$ is the dimension of $\cX$ and $l_{1},\dots,l_{s}$ are positive
integers. The orbifold pseudo-effective cone of $\cX$, denoted by
$\PEff_{\orb}(\cX)$, is a cone in the orbifold Néron--Severi space
of $\cX$, a modification of the usual Néron--Severi space by incorporating
twisted sectors. Let $\pi_{0}^{*}(\cJ_{0}\cX)$ be the finite set
of twisted sectors of $\cX$. The orbifold Néron--Severi space is
defined as 
\[
\N_{\orb}^{1}(\cX)=\N^{1}(\cX)\oplus\bigoplus_{\cY\in\pi_{0}^{*}(\cJ_{0}\cX)}\RR u_{\cY},
\]
where $\N^{1}(\cX)$ is the usual Néron--Severi space and $u_{\cY}$
are indeterminants corresponding to twisted sectors $\cY$ respectively.
Note that there is a one-to-one correspondence between twisted sectors
and nonzero elements of a certain subset of the finitely generated
abelian group $N$ (an input of the stacky fan $\bSigma$) called
the box. Using this correspondence and the structure of the stacky
fan, we can associate a non-negative rational number $a_{\rho}(\cY)$
to each pair $(\rho,\cY)$ of $\rho\in\Sigma(1)$ and $\cY\in\pi_{0}^{*}(\cJ_{0}\cX)$
in an explicit way. The main result of this paper is stated as follows:
\begin{thm}[Corollary \ref{cor:main}]
\label{thm:main-intro}Let $\cE_{\rho}\subset\cX$ be the torus-invariant
prime divisors on $\cX$ corresponding to $\rho\in\Sigma(1)$ respectively
and let $[\cE_{\rho}]$ denote their numerical classes. Then, the
orbifold pseudo-effective cone $\PEff_{\orb}(\cX)$ of $\cX$ is generated
by the $\#(\Sigma(1)\cup\pi_{0}^{*}(\cJ_{0}\cX))$ elements 
\[
[\cE_{\rho}]-\sum_{\cY\in\pi_{0}^{*}(\cJ_{0}\cX)}a_{\rho}(\cY)u_{\cY}\quad(\rho\in\Sigma(1))
\]
and 
\[
u_{\cY}\quad(\cY\in\pi_{0}^{*}(\cJ_{0}\cX)).
\]
\end{thm}

\begin{example}
Let $\cX$ be the weighted projective stack $\cP(2,1)$. This is the
projective line given with a stacky point at $\infty$ with automorphism
group $\mu_{2}$. Thus, there exists a unique twisted sector $\cY$,
lying over $\infty$. The corresponding fan has two one-dimensional
rays $\rho_{0}$ and $\rho_{\infty}$ that corresponds to the points
$0$ and $\infty$, respectively. The orbifold Néron--Severi space
$\N_{\orb}^{1}(\cX)$ is of dimension two and has a basis $[\cE_{\rho_{0}}]=2[\cE_{\rho_{\infty}}]$,
$u_{\cY}$. Theorem \ref{thm:main-intro} says that the orbifold pseudo-effective
cone $\PEff_{\orb}(\cX)$ is generated by the three elements $[\cE_{\rho_{0}}]$,
$[\cE_{\rho_{\infty}}]-\frac{1}{2}u_{\cY}$, and $u_{\cY}$. But,
the first element $[\cE_{\rho_{0}}]$ is obviously redundant and $\PEff_{\orb}(\cX)$
is, in fact, generated by $[\cE_{\rho_{\infty}}]-\frac{1}{2}u_{\cY}$,
and $u_{\cY}$.
\end{example}

Our strategy of the proof is to describe the dual cone of $\PEff_{\orb}(\cX)$,
denoted by $\Mov_{1,\orb}(\cX)$. It is defined to be the closure
of the cone generated by the orbifold numerical classes of covering
families of stacky curves. A difficult part in getting a desired description
of $\Mov_{1,\orb}(\cX)$ is to construct a covering family of stacky
curves that induces a prescribed orbifold numerical class. We will
do this by adapting Payne's argument in \cite{payne2006stablebase}
for our stacky setting; we construct a stacky curve having a desired
numerical property by combining several one-parameter subgroups in
an appropriate way. 

Note that Theorem \ref{thm:main-intro} holds over an arbitrary field
of characteristic zero. However, we may assume that the base field
is algebraically closed, since the action of the absolute Galois group
of $k$ on $\N_{\orb}^{1}(\cX)$ is trivial (since $\cX$ is a \emph{split}
toric stack) and hence the orbifold pseudo-effective cone does not
change by passing to an algebraic closure of the base field. For this
reason, we will assume that the base field $k$ is algebraically closed
in the text of the paper. This simplifies our notation, since we can
then focus on $k$-points of schemes and stacks at most parts of arguments
rather than treating $L$-points for various extensions $L/k$.

\subsection*{Acknowledgements}

We would like to thank Sho Tanimoto for helpful discussions. Ratko
Darda has received a funding from the European Union’s Horizon 2023
research and innovation programme under the Maria Skłodowska-Curie
grant agreement 101149785. Takehiko Yasuda was supported by JSPS KAKENHI
Grant Numbers JP21H04994, JP23H01070 and JP24K00519.

\section{Toric stacks and stacky fans}

From now on, we assume that $k$ is an algebraically closed field
of characteristic zero. However, we note again that Theorem \ref{thm:main-intro}
holds over any field of characteristic zero, as explained in Introduction. 

We recall the notation on toric stacks from \cite{darda2023themanin}
with some modifications. There are several versions of toric stacks
in the literature. We use the version introduced by Borisov--Chen--Smith
\cite{borisov2004theorbifold}. Let $N$ be a finitely generated abelian
group of rank $d$. We fix a decomposition $N=N^{\rig}\oplus N_{\tor}$
of $N$ into a free abelian group $N^{\rig}$ of rank $d$ and a torsion
group $N_{\tor}$. We also fix an isomorphism $N_{\tor}=\prod_{i=1}^{s}\ZZ/l_{i}\ZZ$.
Let $N_{\RR}:=N\otimes\RR$, which is a $\RR$-vector space of dimension
$d$ and naturally identified with $N_{\RR}^{\rig}:=N^{\rig}\otimes\RR$. 

Let $\Sigma$ be a complete simplicial fan in $N_{\RR}$, which corresponds
to a complete, normal and $\QQ$-factorial toric variety $X=X_{\Sigma}$
of dimension $d$ over the base field $k$. Let $\bigoplus_{\rho\in\Sigma(1)}\ZZ\bv_{\rho}$
be a free abelian group of rank $\#\Sigma(1)$ with $\Sigma(1)$ denoting
the set of 1-dimensional cones in $\Sigma$ and let 
\[
\beta\colon\bigoplus_{\rho\in\Sigma(1)}\ZZ\bv_{\rho}\to N
\]
be a homomorphism such that the triple $\bSigma=(\Sigma,N,\beta)$
is a stacky fan in the sense of \cite{borisov2004theorbifold}. Let
$\cX=\cX_{\bSigma}$ be the toric stack corresponding to the stacky
fan $\bSigma$. This is a smooth and proper DM stack of dimension
$d$ over $k$. There exists a natural morphism $\cX\to X$, which
makes $X$ a coarse moduli space of $\cX$. 

Let $\beta^{\rig}$ denote the composition of $\beta$ with the quotient
map $N\to N^{\rig}$. The triple $\bSigma^{\rig}:=(\Sigma,N^{\rig},\beta^{\rig})$
is again a stacky fan. The associated toric stack $\cX^{\rig}:=\cX_{\bSigma^{\rig}}$
has a smooth and proper DM stack which has trivial generic stabilizer.
There exists a natural morphism $\cX\to\cX^{\rig}$, which is the
rigidification of $\cX$ with respect to the dominating components
of the inertia stack $\I\cX$. We also have a natural morphism $\cX^{\rig}\to X$,
which allows us to regard $X$ also as a coarse moduli space of $\cX^{\rig}$.
The composition $\cX\to\cX^{\rig}\to X$ is identical to the morphism
$\cX\to X$ above. 

Let $T\cong\GG_{m}^{d}$ be the algebraic torus corresponding to $N^{\rig}$
and let $\cT=T\times\prod_{i=1}^{s}\B\mu_{l_{i}}$ be the stacky torus
corresponding to $N$. The toric stack $\cX$ contains $\cT$ as an
open substack. From \cite[Definition 3.1 and Theorem 7.24]{fantechi2010smoothtoric},
we have an action of $\cT$ on $\cX$. 

For each $\rho\in\Sigma(1)$, let $b_{\rho}:=\beta^{\rig}(\bv_{\rho})\in N^{\rig}$
and let $w_{\rho}$ be the primitive element of $N^{\rig}\cap\rho$,
where we identify $N^{\rig}$ with a subgroup of $N_{\RR}$ in the
obvious way. There exists a unique $c_{\rho}\in\ZZ_{>0}$ with $b_{\rho}=c_{\rho}w_{\rho}$.
Let $\cE_{\rho},\cE_{\rho}^{\rig},E_{\rho}$ be the prime divisors
on $\cX$, $\cX^{\rig}$ and $X$ respectively that correspond to
$\rho$. In particular, they are irreducible and reduced closed substacks
of $\cX$, $\cX^{\rig}$ and $X$ respectively and there are natural
morphisms $\cE_{\rho}\to\cE_{\rho}^{\rig}\to E_{\rho}$. The pullbacks
of $E_{\rho}$ to $\cX^{\rig}$ and $\cX$ are $c_{\rho}\cE_{\rho}$
and $c_{\rho}\cE_{\rho}^{\rig}$.
\begin{rem}
We should not confuse $\cE_{\rho}^{\rig}$ with the rigidification
of $\cE_{\rho}$ with respect to dominant components of its inertia
stack. Since the stack $\cE_{\rho}^{\rig}$ can have non-trivial generic
stabilizer, it is only the rigidification of $\cE_{\rho}$ with respect
to \emph{some }dominant components of the inertia stack. 
\end{rem}

\section{Sectors and boxes}

The stack of twisted 0-jets of $\cX$, denoted by $\cJ_{0}\cX$, is
the stack parametrizing representable morphisms $\B\mu_{l}\to\cX$,
$l\in\ZZ_{>0}$. It is again a smooth and proper Deligne--Mumford
stack over $k$. The original $\cX$ is regarded as a connected component
of $\cJ_{0}\cX$. Let $\pi_{0}(\cJ_{0}\cX)$ be the set of connected
component of $\cJ_{0}\cX$ and let $\pi_{0}^{*}(\cJ_{0}\cX):=\pi_{0}(\cJ_{0}\cX)\setminus\{\cX\}$.
These are finite sets. An element of $\pi_{0}(\cJ_{0}\cX)$ (resp.
$\pi_{0}^{*}(\cJ_{0}\cX)$) is called a sector (resp. a twisted sector)
of $\cX$. See \cite[Section 2.1]{darda2024thebatyrevtextendashmanin}
for more details.

For $\by\in N^{\rig}$, there are rational numbers $a_{\rho}(\by)\in\QQ_{\ge0}$,
$\rho\in\Sigma(1)$ such that $a_{\rho}(\by)>0$ if and only if $\rho$
is a face of the minimal cone $\sigma\in\Sigma$ containing $\by$
and such that we can write 
\[
\by=\sum_{\rho\in\Sigma(1)}a_{\rho}(\by)b_{\rho}.
\]
See \cite[Section 3.1.1]{darda2023themanin}. We define 
\[
\Box^{\rig}(\Sigma):=\{\by\in N^{\rig}\mid\forall\rho\in\Sigma(1),\,a_{\rho}(\by)<1\}.
\]
We define $\Box(\Sigma)$ to be the preimage of $\Box^{\rig}(\Sigma)$
by the projection $N=N^{\rig}\oplus N_{\tor}\to N^{\rig}$, which
is identified with $\Box^{\rig}(\Sigma)\times N_{\tor}$ in the obvious
way. There are canonical identifications 
\[
\pi_{0}(\cJ_{\infty}\cX)=\Box(\Sigma)\text{ and }\pi_{0}(\cJ_{\infty}\cX^{\rig})=\Box^{\rig}(\Sigma).
\]
See \cite[Proposition 4.7]{borisov2004theorbifold} and \cite[Lemma 3.1.4]{darda2023themanin}.

We have a map 
\[
q\colon N^{\rig}\to\Box^{\rig}(\Sigma),\quad\by=\sum_{\rho\in\Sigma(1)}a_{\rho}(\by)b_{\rho}\mapsto\sum_{\rho\in\Sigma(1)}\{a_{\rho}(\by)\}b_{\rho}.
\]
Here, for a rational number $s$, $\{s\}$ denotes the fractional
part $s-\lfloor s\rfloor$ of $s$. We extend this map to
\[
q\times\id_{G^{D}}\colon N^{\rig}\times G^{D}\to\Box(\Sigma)=\Box^{\rig}(\Sigma)\times G^{D},
\]
which we denote again by $q$.

\section{Exact sequences}

\subsection{``Non-orbifold'' exact sequences}

We first recall two standard exact sequences which are dual to each
other and involve the spaces of divisors and 1-cycles, $M_{\RR}$
and $N_{\RR}$. Consider the following two $\RR$-vector spaces 
\begin{gather*}
U=\bigoplus_{\rho\in\Sigma(1)}\RR\bu_{\rho}\text{ and }V=\bigoplus_{\rho\in\Sigma(1)}\RR\bv_{\rho}.
\end{gather*}
We regard $U$ and $V$ as dual spaces to each other by the pairing
determined by $\langle\bu_{\rho},\bv_{\rho'}\rangle=\delta_{\rho,\rho'}$
(Kronecker's delta). Let 
\[
M:=\Hom(N,\ZZ)=\Hom(N^{\rig},\ZZ)\cong\ZZ^{d}
\]
and let $M_{\RR}:=M\otimes\RR\cong\RR^{d}.$ Let $\N^{1}(X)$ denote
the Néron-Severi space of $X$ defined over $\RR$ and let $\N_{1}(X)$
be the space of numerical classes of $1$-cycles defined over $\RR$.
These are $\RR$-vectors spaces dual to each other and their dimension
is the Picard number of $X$. 

We consider the map 
\[
\alpha\colon M_{\RR}\to U,\quad\bm\mapsto\sum_{\rho\in\Sigma(1)}\langle w_{\rho},\bm\rangle\bu_{\rho}
\]
and the map 
\[
\lambda\colon U\to\N^{1}(X),\quad\bu_{\rho}\mapsto[E_{\rho}],
\]
where $[E_{\rho}]$ is the numerical class of $E_{\rho}$. As is well-known
(for example, see \cite[Propositions 6.3.15 and 6.4.1]{cox2011toricvarieties}),
these maps fit into the short exact sequence:
\begin{equation}
0\to M_{\RR}\xrightarrow{\alpha}U\xrightarrow{\lambda}\N^{1}(X)\to0.\label{eq:ex1}
\end{equation}
Let us consider its dual exact sequence,
\begin{equation}
0\to\N_{1}(X)\xrightarrow{\mu}V\xrightarrow{\beta'}N_{\RR}\to0.\label{eq:ex2}
\end{equation}
Here, the map $\mu$ is given by $\mu(h)=\sum_{\rho\in\Sigma(1)}(\mu,E_{\rho})\bv_{\rho}$
and the map $\beta'$ is given by $\beta'(\bv_{\rho})=w_{\rho}$.
Thus, $\N_{1}(X)$ is identified with $\Ker(\beta')$, which is explicitly
described as
\[
\Ker(\beta')=\left\{ \sum_{\rho\in\Sigma(1)}a_{\rho}\bv_{\rho}\bigmid\sum_{\rho\in\Sigma(1)}a_{\rho}w_{\rho}=0\right\} .
\]

\begin{rem}
Let us denote the map $V\to N_{\RR}$ obtained by tensoring $\beta$
with $\RR$ again by $\beta$ and let $c\colon V\to V$ be the linear
map with $\bv_{\rho}\mapsto c_{\rho}\bv_{\rho}$. Then, $\beta'$
and $\beta$ is related by $\beta=c\circ\beta'$.
\end{rem}

\subsection{``Orbifold'' exact sequences}

In what follows, we often treat elements of $\Sigma(1)$ and $\pi_{0}^{*}(\cJ_{0}\cX)$
equally by denoting them by $\eta$. On the other hand, when we treat
them separately, we denote elements of $\Sigma(1)$ by $\rho$ and
elements of $\pi_{0}^{*}(\cJ_{0}\cX)$ by $\cY$. 

Next we construct the orbifold versions of the exact sequences (\ref{eq:ex1})
and (\ref{eq:ex2}) obtained above. Consider the following two $\RR$-vector
spaces 
\begin{gather*}
U_{\orb}=\bigoplus_{\eta\in\Sigma(1)\cup\pi_{0}^{*}(\cJ_{0}\cX)}\RR\bu_{\eta}\text{ and }V_{\orb}=\bigoplus_{\eta\in\Sigma(1)\cup\pi_{0}^{*}(\cJ_{0}\cX)}\RR\bv_{\eta}.
\end{gather*}
We regard $U_{\orb}$ and $V_{\orb}$ as dual spaces to each other
by the pairing $\langle\bu_{\eta},\bv_{\eta'}\rangle=\delta_{\eta,\eta'}$.
Let 
\[
\N_{\orb}^{1}(\cX):=\N^{1}(X)\oplus\bigoplus_{\cY\in\pi_{0}^{*}(\cJ_{0}\cX)}\RR u_{\cY}
\]
be the orbifold Néron--Severi space defined over $\RR$. Similarly,
we define
\[
\N_{1,\orb}(\cX):=\N_{1}(X)\oplus\bigoplus_{\cY\in\pi_{0}^{*}(\cJ_{0}\cX)}\RR v_{\cY}.
\]
The intersection pairing between $\N_{\orb}^{1}(\cX)$ and $\N_{1,\orb}(\cX)$
is defined as the component-wise pairing. It makes them dual to each
other. Taking the direct sum of (\ref{eq:ex1}) and 
\[
0\to0\to\bigoplus_{\cY\in\pi_{0}^{*}(\cJ_{0}\cX)}\RR\bu_{\cY}\xrightarrow{\bu_{\cY}\mapsto u_{^{\cY}}}\bigoplus_{\cY\in\pi_{0}^{*}(\cJ_{0}\cX)}\RR u_{\cY}\to0
\]
gives the short exact sequence
\[
0\to M_{\RR}\xrightarrow{\alpha_{\orb}}U_{\orb}\xrightarrow{\lambda_{\orb}}\N^{1}(\cX)_{\orb}\to0.
\]
As the dual of the last sequence, we get the short exact sequence
\[
0\to\N_{1,\orb}(\cX)\xrightarrow{\mu_{\orb}}V_{\orb}\xrightarrow{\beta'_{\orb}}N_{\RR}\to0.
\]
The map $\beta'_{\orb}\colon V_{\orb}\to N_{\RR}$ sends $\bv_{\rho}$
to $w_{\rho}$ the primitive lattice vector in $N^{\rig}$ on $\rho$
and $\bv_{\cY}$ to $0$. Thus, $\N_{1,\orb}(\cX)$ is identified
with the kernel of $\beta'_{\orb}$, which is described as
\begin{equation}
\begin{aligned}\Ker(\beta_{\orb}') & =\left\{ \sum_{\eta\in\Sigma(1)\cup\pi_{0}^{*}(\cJ_{0}\cY)}a_{\eta}\bv_{\eta}\in V_{\orb}\bigmid\sum_{\rho\in\Sigma(1)}a_{\rho}w_{\rho}=0\right\} \\
 & =\Ker(\beta')\oplus\bigoplus_{\cY\in\pi_{0}^{*}(\cJ_{0}\cX)}\RR\bv_{\cY}.
\end{aligned}
\label{eq:Ker-orb}
\end{equation}

\section{Orbifold numerical classes of stacky curves}

\subsection{$k\protect\tpars$-points and twisted arcs}

Let $D=\Spec k\tbrats$ be the formal disk. For $l\in\ZZ_{>0}$, let
$\cD^{l}:=[\Spec k\llbracket t^{1/l}\rrbracket/\mu_{l}]$, a twisted
formal disk. There exists a natural morphism $\cD^{l}\to D$ making
$D$ the coarse moduli space of $\cD^{l}$. This morphism is an isomorphism
over the punctured formal disk $D^{*}:=\Spec k\tpars\subset D$. Thus,
we can regard $D^{*}$ also as an open substack of $\cD^{l}$. A \emph{twisted
arc} of $\cX$ is a representable morphism of the form $\cD^{l}\to\cX$
for some $l\in\ZZ_{>0}$. A twisted arc induces a $k\tpars$-point
of $\cX$ as the composition 
\[
D^{*}\hookrightarrow\cD^{l}\to\cX.
\]
Conversely, for every $k\tpars$-point $\gamma\colon D^{*}\to\cX$,
there exists a unique twisted arc $\cD^{l}\to\cX$ inducing $\gamma$.
Thus, we have a one-to-one correspondence between $\cX\langle k\tpars\rangle$
and $(\cJ_{\infty}\cX)\langle k\rangle$. Here the notation $\cX\langle L\rangle$
means the set of isomorphism classes in the groupoid $\cX(L)$ of
$L$-points of the stack $\cX$. 

In what follows, we often identify a $k\tpars$-point with its corresponding
twisted arc and denote them by the same symbol, which is often $\gamma$. 

\subsection{The residue map}
\begin{defn}
Given a twisted arc $\gamma\colon\cD^{l}\to\cX$, composing it with
the natural closed immersion $\B\mu_{l}\hookrightarrow\cD^{l}$, we
get a $k$-point of $\cJ_{0}\cX$. We define the\emph{ sector associated
to} $\gamma$ to be the connected component of $\cJ_{0}\cX$ containing
this $k$-point and denote it by $\cY_{\gamma}$. We also call it
the \emph{sector associated to} $\gamma$. The map 
\[
(\cJ_{\infty}\cX)\langle k\rangle\to\pi_{0}(\cJ_{0}\cX),\quad\gamma\to\cY_{\gamma}
\]
is called the \emph{residue map}. 
\end{defn}

From \cite[\S 2.1]{darda2023themanin}, we have the map 
\[
\log_{\cT}\colon\cT\langle k\tpars\rangle\to N.
\]
Via the identification $\cT=\GG_{m}^{d}\times\prod_{i=1}^{s}\B\mu_{l_{i}}$,
the above map is identified with 
\begin{align*}
(k\tpars^{*})^{d}\times\prod_{i=1}^{s}k\tpars^{*}/(k\tpars^{*})^{l_{i}} & \to\ZZ^{d}\times\prod_{i=1}^{s}\ZZ/l_{i}\ZZ\\
((h_{i})_{1\le i\le d},([g_{i}])_{1\le i\le s}) & \mapsto((\ord h_{i})_{1\le i\le d},([\ord g_{i}])_{1\le i\le s}).
\end{align*}
From \cite[Lemmas 3.1.6 and 3.1.9]{darda2023themanin}, the residue
map $\psi\colon\cT\langle k\tpars\rangle\to\pi_{0}(\cJ_{0}\cX)$ is
identical to 
\[
q\circ\log_{\cT}\colon\cT\langle k\tpars\rangle\to N^{\rig}\times G^{D}\xrightarrow{q}\Box^{\rig}(\Sigma)\times G^{D}=\Box(\Sigma).
\]

\subsection{Orbifold numerical classes}
\begin{defn}
By a \emph{stacky curve }on $\cX$, we mean a representable morphism
$f\colon\cC\to\cX$ such that $\cC$ is an irreducible smooth and
proper DM stack over $k$ having the trivial generic stabilizer. 
\end{defn}

Let $f\colon\cC\to\cX$ be a stacky curve. For each $k$-point $z\in\cC$,
the formal completion $\cC_{z}$ of $\cC$ at $z$ is isomorphic to
$\cD^{l}:=[\Spec k\llbracket t^{1/l}\rrbracket/\mu_{l}]$. The composition
\[
\gamma\colon\cD^{l}\xrightarrow{\sim}\cC_{z}\to\cX
\]
is a twisted arc. The sector $\cY_{\gamma}$ associated to $\gamma$
is independent of the chosen isomorphism and depends only on $f$
and $z$. We call it the \emph{sector of $f$ at $z$ }and denote
it by $\cY_{f,z}$ or just by $\cY_{z}$. 
\begin{defn}[{\cite[Remark 8.9]{darda2024thebatyrevtextendashmanin}}]
\label{def:orb-num-cl}For a stacky curve $f\colon\cC\to\cX$, we
define its \emph{orbifold numerical class} to be 
\[
[f]_{\orb}:=[\overline{f}]+\sum_{\substack{z\in\cC\langle k\rangle\\
\text{stacky points}
}
}v_{\cY_{z}}\in\N_{1,\orb}(\cX),
\]
where $\overline{f}$ is the morphism $\overline{\cC}\to X$ derived
from $f$ with $\overline{\cC}$ the coarse moduli space of $\cC$,
$\cY_{z}$ is the twisted sector associated to $f$ and $x$. 
\end{defn}

\subsection{Localization of orbifold numerical classes}

We will see that when a stacky curve $f\colon\cC\to\cX$ maps the
generic point into $\cT$, then the orbifold numerical class $[f]_{\orb}$
is written as the summation of ``local classes'' $[f_{z}]_{\orb}$,
$z\in\cC\langle k\rangle$. 

For $\rho\in\Sigma(1)$, let $\cI_{\rho}\subset\cO_{\cX}$ be the
coherent ideal sheaf defining the prime divisor $\cE_{\rho}\subset\cX$.
For a twisted arc $\gamma\colon\cD^{l}\to\cX$, consider the composite
morphism
\[
\widetilde{\gamma}\colon\Spec k\llbracket t^{1/l}\rrbracket\to\cD^{l}\to\cX.
\]
The pull-back $(\widetilde{\gamma})^{-1}\cI_{\rho}\subset k\llbracket t^{1/l}\rrbracket$
is an ideal of the form $(t^{n/l})$, $n\in\ZZ_{\ge0}\cup\{\infty\}$.
Here we followed the convention $(t^{\infty})=(0)$. If the corresponding
$k\tpars$-point lies in $\cT$, then $n$ is finite. 
\begin{defn}
We define the \emph{intersection number} of $\gamma$ and $\cE_{\rho}$
to be
\[
(\gamma,\cE_{\rho}):=n/l
\]
with $n$ and $l$ as above. We define $(\gamma,E_{\rho}):=c_{\rho}(\gamma,\cE_{\rho})$.
\end{defn}

\begin{defn}
For $\rho\in\Sigma(1)$, let $\bb_{\rho}:=c_{\rho}\bv_{\rho}\in V_{\orb}$.
We define the \emph{orbifold numerical class }of $\gamma\in\cT(k\tpars)$
to be 
\[
[\gamma]_{\orb}:=\sum_{\rho\in\Sigma(1)}(\gamma,\cE_{\rho})\bb_{\rho}+\bv_{\cY_{\gamma}}=\sum_{\rho\in\Sigma(1)}(\gamma,E_{\rho})\bv_{\rho}+\bv_{\cY_{\gamma}}\in V_{\orb}.
\]
\end{defn}

\begin{prop}
\label{prop:orb-pointwise}Let $f\colon\cC\to\cX$ be a stacky curve
such that the generic point of $\cC$ mapping into $\cT$. Then, we
have 
\[
[f]_{\orb}=\sum_{z\in\cC\langle k\rangle}[f_{z}]_{\orb}.
\]
\end{prop}

\begin{proof}
We first note that since $[f_{z}]_{\orb}=0$ all but finitely many
points $z$, the sum in the proposition is a finite sum. We have
\begin{align*}
\sum_{z\in\cC\langle k\rangle}[f_{z}]_{\orb} & =\sum_{z\in\cC\langle k\rangle}\sum_{\rho\in\Sigma(1)}(f_{z},E_{\rho})\bb_{\rho}+\sum_{\substack{z\in\cC\langle k\rangle\\
\text{stacky points}
}
}\bv_{\cY_{z}}\\
 & =[\overline{f}]+\sum_{\substack{z\in\cC\langle k\rangle\\
\text{stacky points}
}
}\bv_{\cY_{z}}\\
 & =[f]_{\orb}.
\end{align*}
\end{proof}
For $\gamma\in\cT(k\tpars)$, let $u:=\log_{T}(\gamma)\in N^{\rig}$
and let $\sigma\in\Sigma$ be a $d$-dimensional cone containing $u$.
The open substack $\cU(\sigma)\subset\cX$ corresponding to $\sigma$
is of the form $[\AA_{k}^{d}/G_{\sigma}]$ for a finite abelian group
$G_{\sigma}$. From the proof of \cite[3.1.9]{darda2023themanin},
the twisted arc $[\Spec k\llbracket t^{1/l}\rrbracket/\mu_{l}]\to\cX$
that corresponds to $\gamma$ factors as 
\[
[\Spec k\llbracket t^{1/l}\rrbracket/\mu_{l}]\to\cU(\sigma)=[\AA_{k}^{d}/G_{\sigma}].
\]
For a unique embedding $\iota\colon\mu_{l}\hookrightarrow G_{\sigma}$,
there exists a (non-unique) lift of this morphism to an $\iota$-equivariant
morphism
\[
\widetilde{\gamma}\colon\Spec k\llbracket t^{1/l}\rrbracket\to\AA_{k}^{d}.
\]
Let $\rho_{1},\dots,\rho_{d}$ be the one-dimensional faces of $\sigma$,
let $\by\in\Box^{\rig}(\sigma)$ be the image of $u$ and let $a_{\rho_{i}}^{\sigma}(\by)=a_{i}\in\frac{1}{l}\ZZ$.
We have 
\[
(\widetilde{\gamma})^{*}x_{i}=t^{a_{i}}f\quad(f\in k\tbrats).
\]
In particular, 
\[
\ord(\widetilde{\gamma})^{*}x_{i}\in a_{i}+\ZZ_{\ge0}\subset\frac{1}{l}\ZZ_{\ge0}.
\]

With the notation above, if $\rho_{1},\dots,\rho_{d}\in\Sigma(1)$
are such that $\cE_{\rho_{i}}$, $1\le i\le d$, are defined by $x_{i}=0$
respectively, then 
\[
(\gamma,\cE_{\rho_{i}})=\ord(\widetilde{\gamma})^{*}x_{i}=a_{i}+\ord f.
\]
For $\rho\in\Sigma(1)\setminus\{\rho_{1},\dots,\rho_{d}\}$, we have
$(\gamma,\cE_{\rho})=0$. 

With the notation above, we can write
\[
[\gamma]_{\orb}=\sum_{i=1}^{d}a_{i}\bb_{\rho_{i}}+\bv_{\cY_{\gamma}}.
\]
Moreover, if $e$ is the image of $\gamma$ by the composite map
\[
\cT\langle k\tpars\rangle\to N\to N_{\tor},
\]
then the sector $\cY_{\gamma}$ corresponds to the element 
\[
\left(\sum_{i=1}^{d}\{a_{i}\}b_{\rho_{i}},e\right)\in\Box(\Sigma)=\Box^{\rig}(\Sigma)\times N_{\tor}.
\]

\begin{defn}
For each $\eta\in\Sigma(1)\cup\pi_{0}^{*}(\cJ_{0}\cX)$, we define
an element $\Xi_{\eta}\in V_{\orb}$ as follows: 
\[
\Xi_{\eta}:=\begin{cases}
\bb_{\rho} & (\eta=\rho\in\Sigma(1))\\
\bv_{\cY}+\sum_{\rho\in\Sigma(1)}a_{\rho}(\cY)\bb_{\rho} & (\eta=\cY\in\pi_{0}^{*}(\cJ_{0}\cX)).
\end{cases}
\]
For the non-twisted sector $\cX\in\pi_{0}^{*}(\cJ_{0}\cX)$, we put
$\Xi_{\cX}:=0\in V_{\orb}$. 
\end{defn}

\begin{lem}
\label{lem:orb-class-included}For $\gamma\in\cT(k\tpars)$, we have
\[
[\gamma]_{\orb}\in\Xi_{\cY_{\gamma}}+\sum_{\rho\in\Sigma(1)}\ZZ_{\ge0}\Xi_{\rho}\subset\sum_{\eta\in\Sigma(1)\cup\pi_{0}^{*}(\cJ_{0}\cX)}\ZZ_{\ge0}\Xi_{\eta}.
\]
\end{lem}

\begin{proof}
We have
\begin{align*}
[\gamma]_{\orb} & =\sum_{\rho\in\Sigma(1)}(\gamma,\cE_{\rho})\bb_{\rho}+\bv_{\cY_{\gamma}}\\
 & =\left(\sum_{\rho\in\Sigma(1)}\{(\gamma,\cE_{\rho})\}\bb_{\rho}+\bv_{\cY_{\gamma}}\right)+\sum_{\rho\in\Sigma(1)}\lfloor(\gamma,E_{\rho})\rfloor\bb_{\rho}\\
 & =\Xi_{\cY_{\gamma}}+\sum_{\rho\in\Sigma(1)}\lfloor(\gamma,E_{\rho})\rfloor\bb_{\rho}.
\end{align*}
which is an element of $\Xi_{\cY_{\gamma}}+\sum_{\rho\in\Sigma(1)}\ZZ_{\ge0}\Xi_{\rho}$.
\end{proof}
\begin{cor}
\label{cor:orb-class}Let $f\colon\cC\to\cX$ be a stacky curve such
that the generic point of $\cC$ mapping into $\cT$. Then, we have
\[
[f]_{\orb}\in\sum_{\eta\in\Sigma(1)\cup\pi_{0}^{*}(\cJ_{0}\cX)}\ZZ_{\ge0}\Xi_{\eta}.
\]
\end{cor}

\begin{proof}
This is a direct consequence of Lemma \ref{lem:orb-class-included}
and Proposition \ref{prop:orb-pointwise}. 
\end{proof}

\subsection{The orbifold class derived by a one-parameter subgroup}

For each $b\in N=N^{\rig}\times N_{\tor}$, we have the associated
morphism 
\[
\phi_{b}\colon\GG_{m}\to\cT.
\]
For $\rho\in\Sigma(1)$, let $b_{\rho}$ denote $\beta(\bv_{\rho})$.
The natural map $k(t)\hookrightarrow k\tpars$ induces the $k\tpars$-point
\[
\widehat{\phi_{b}}\colon\Spec k\tpars\to\GG_{m}\to\cT.
\]

\begin{lem}
We have 
\[
\left[\widehat{\phi_{b}}\right]_{\orb}=\sum_{\rho\in\Sigma(1)}a_{\rho}(b)\bb_{\rho}+\bv_{q(b)}.
\]
\end{lem}

\begin{proof}
By construction, 
\[
\log_{\cT}(\widehat{\phi_{b}})=b=\left(\sum_{\rho\in\Sigma(1)}a_{\rho}(b)\bb_{\rho},e\right)\in N^{\rig}\times N_{\tor},
\]
with $g\in G^{D}$. Thus, the sector associated to $\widehat{\phi_{b}}$
is 
\[
\cY_{\widehat{\phi_{b}}}=q(b)=\left(\sum_{\rho\in\Sigma(1)}\{a_{\rho}(b)\}\bb_{\rho},e\right)\in\Box(\Sigma).
\]
Thus,
\[
\left[\widehat{\phi_{b}}\right]_{\orb}=\sum_{\rho\in\Sigma(1)}a_{\rho}(b)\bb_{\rho}+\bv_{q(b)}.
\]
\end{proof}
\begin{cor}
\label{cor:Xi}For $\eta\in\Sigma(1)\cup\pi_{0}^{*}(\cJ_{0}\cX)$,
we have $\left[\widehat{\phi_{b_{\eta}}}\right]_{\orb}=\Xi_{\eta}$. 
\end{cor}

\begin{proof}
This is a direct consequence of the last lemma. 
\end{proof}

\section{The ``action'' of $\protect\J_{\infty}\protect\cT$ on $\protect\cJ_{\infty}\protect\cX$\label{sec:action}}

From \cite[Definition 3.1 and Theorem 7.24]{fantechi2010smoothtoric},
we have an action 
\[
\cT\times\cX\to\cX
\]
of the stacky torus $\cT$ on the toric stack $\cX$. 

Let $\J_{\infty}\cT$ denote the stack of (untwisted) arcs on $\cT$.
A $k$-point of $\J_{\infty}\cT$ corresponds to a morphism
\[
D:=\Spec k\tbrats\to\cT.
\]
Let $\cJ_{\infty}\cX$ denote the stack of twisted arcs on $\cX$.
A $k$-point of $\cJ_{\infty}\cX$ corresponds to a representable
morphism
\[
\cD^{l}:=[\Spec k\llbracket t^{1/l}\rrbracket/\mu_{l}]\to\cX.
\]

Let $\delta\colon D\to\cT$ be an arc and let $\gamma\colon\cD^{l}\to\cX$
be a twisted arc. Let $\delta\cdot\gamma\colon\cD^{l}\to\cX$ be the
twisted arc obtained as the composition
\[
\delta\cdot\gamma\colon\cD^{l}\to D\times\cD^{l}\to\cT\times\cX\to\cX.
\]
Let $\cX_{D}:=\cX\times D$ and $\cT_{D}:=\cT\times D$. The arc $\delta$
also defines a $\D$-automorphism $\alpha_{\delta}\colon\cX_{\D}\to\cX_{\D}$.
The twisted arc $\delta\cdot\gamma$ is isomorphic to $\alpha_{\delta}\circ\gamma$,
which follows from the following 2-commutative diagram.
\[
\xymatrix{\\\cD^{l}\ar[r]\ar[d]^{\gamma}\ar@/^{4pc}/[rrd]^{\delta\cdot\gamma} & D\times_{D}\cD^{l}\ar[d]^{(\delta,\gamma)}\\
\cX_{D}\ar[r]\ar[d]\ar@/_{1pc}/[rr]\sb(0.3){\alpha_{\delta}} & \cT_{D}\times_{D}\cX_{D}\ar[r]\ar[d] & \cX_{D}\ar[d]\\
D\ar[r]^{\delta}\ar@/_{1pc}/[rr]_{\id_{\D}} & \cT_{D}\ar[r] & D
}
\]
To see that $\delta\cdot\gamma$ is indeed a twisted arc, we need
to see that $\delta\cdot\gamma$ is a representable morphism. To show
this, we take a quasi-inverse arc $\delta^{-1}\colon D\to\cT$ of
$\delta$. The induced morphism $\alpha_{\delta^{-1}}$ is a quasi-inverse
to $\alpha_{\delta}$. This implies that $\alpha_{\delta}$ is representable,
which shows that $\delta\cdot\gamma=\alpha_{\delta}\circ\gamma$ is
also representable. 
\begin{lem}
\label{lem:intnum-same}For every $\rho\in\Sigma(1)$, $(\delta\cdot\gamma,\cE_{\rho})=(\gamma,\cE_{\rho})$.
\end{lem}

\begin{proof}
Let $\cE_{\rho,D}\subset\cX_{D}$ be the closed substack derived from
$\cE_{\rho}$ by the base change. The isomorphism $\alpha_{\delta}$
preserves $\cE_{\rho,D}$. Thus, 
\[
(\delta\cdot\gamma)^{-1}\cI_{\cE_{\rho}}=\gamma^{-1}((\alpha_{\delta})^{-1}\cI_{\cE_{\rho,D}})=\gamma^{-1}\cI_{\cE_{\rho,D}}.
\]
This shows the lemma.
\end{proof}
\begin{lem}
\label{lem:sector-same}The twisted arcs $\delta\cdot\gamma$ and
$\gamma$ induce the same sector. Namely, $\cY_{\delta\cdot\gamma}=\cY_{\gamma}$.
\end{lem}

\begin{proof}
Let $\delta_{0}\colon\Spec k\to\cT$ and $\gamma_{0}\colon\B\mu_{l}\to\cX$
be the morphisms obtained by restricting $\delta$ and $\gamma$ to
the reduced closed substacks at the unique closed point of the source,
respectively. We get an automorphism $\alpha_{\delta_{0}}\colon\cX\to\cX$
obtained from $\delta_{0}$. We need to show that the two morphisms
$\alpha_{\delta_{0}}\circ\gamma_{0},\gamma_{0}\colon\B\mu_{l}\to\cX$
determines points in the same connected component of $\cJ_{0}\cX$.
Let $e\colon T\to\cT$ be the canonical morphism. We get an automorphism
$\alpha_{e}\colon\cX_{T}\to\cX_{T}$. This induces a morphism 
\[
T\times\cJ_{0}\cX\to\cJ_{0}\cX.
\]
From the construction, the composition 
\[
\cJ_{0}\cX\xrightarrow{(1_{T},\id_{\cJ_{0}\cX})}T\times\cJ_{0}\cX\to\cJ_{0}\cX
\]
is the identity morphism. Therefore, for each sector $\cY\subset\cJ_{0}\cX$,
since $T$ is connected and $\cY$ is geometrically connected, the
image of $T\times\cY$ is $\cY$ and get a morphism $T\times\cY\to\cY$.
Let $\widetilde{\delta_{0}}\colon\Spec k\to T$ be a lift of $\delta_{0}$.
If $\gamma_{0}$ corresponds to a point of $\cY$, then $\alpha_{\delta_{0}}\circ\gamma_{0}$
corresponds to
\[
\Spec k\xrightarrow{(\widetilde{\delta_{0}},\gamma_{0})}T\times\cY\to\cY.
\]
In particular, this is a point of the same sector $\cY$, as desired.
\end{proof}
\begin{cor}
\label{cor:orb-same}We have $[\delta\cdot\gamma]_{\orb}=[\gamma]_{\orb}$.
\end{cor}

\begin{proof}
This follows from Lemmas \ref{lem:intnum-same} and \ref{lem:sector-same}
and Definition \ref{def:orb-num-cl}.
\end{proof}
\begin{rem}
We guess that our construction of $\delta\cdot\gamma$ would be eventually
obtained from an action of $\J_{\infty}\cT$ on $\cJ_{\infty}\cX$
in the sense of \cite[Appendix B]{fantechi2010smoothtoric}. In this
paper, however, we have decided to content ourselves with showing
only properties of $\delta\cdot\gamma$ that we need in the proof
of the main result.
\end{rem}

\section{Orbifold pseudo-effective cones}
\begin{defn}
A \emph{covering family of stacky curves on} $\cX$ is the pair $(\pi\colon\widetilde{\cC}\to S,\widetilde{f}\colon\widetilde{\cC}\to\cX)$
of a smooth morphism $\pi\colon\widetilde{\cC}\to S$ of DM stacks
with $T$ an irreducible $k$-variety and a representable dominant
morphism $\widetilde{f}\colon\widetilde{\cC}\to\cX$ such that for
every $s\in S(k)$, the morphism
\[
\widetilde{f}|_{\pi^{-1}(s)}\colon\pi^{-1}(s)\to\cX
\]
is a stacky curve on $\cX$. 
\end{defn}

\begin{defn}
For a covering family $(\pi\colon\widetilde{\cC}\to S,\widetilde{f}\colon\widetilde{\cC}\to\cX)$
of stacky curves on $\cX$, we define its \emph{orbifold numerical
class, }denoted by $[\widetilde{f}]_{\orb}$, to be $[\widetilde{f}|_{\pi^{-1}(s)}]$
for a general point $s\in S(k)$ (which is well-defined thanks to
\cite[Lemma 8.6]{darda2024thebatyrevtextendashmanin}).
\end{defn}

Let $f\colon\cC\to\cX$ be a stacky curve. From (\ref{eq:Ker-orb}),
$[f]_{\orb}$ lies in $\Ker(\beta_{\orb}')$ and hence is regarded
as an element of $\N_{1,\orb}(\cX)$.
\begin{defn}
We define a cone $\Mov_{1,\orb}(\cX)\subset\N_{1,\orb}(\cX)$ to be
the closure of the cone generated by the classes $[\widetilde{f}]_{\orb}$
of all the covering families $(\pi\colon\widetilde{\cC}\to S,\widetilde{f}\colon\widetilde{\cC}\to\cX)$
of stacky curves on $\cX$. We define the \emph{orbifold pseudo-effective
cone }$\PEff_{\orb}(\cX)\subset\N_{\orb}^{1}(\cX)$ to be the dual
cone of $\Mov_{1,\orb}(\cX)$.
\end{defn}

\begin{lem}
The cone $\Mov_{1,\orb}(\cX)\subset\N_{1,\orb}(\cX)$ is the closure
of the cone generated by by the classes $[f]_{\orb}$ of those stacky
curves $f\colon\cC\to\cX$ that map the generic point into $\cT$.
\end{lem}

\begin{proof}
It is enough to show that the two sets of orbifold numerical classes
are identical; the orbifold numerical classes of covering families
of stacky curves and the ones of stacky curves mapping the generic
point into $\cT$. If $(\pi\colon\widetilde{\cC}\to S,\widetilde{f}\colon\widetilde{\cC}\to\cX)$
is a covering family of stacky curves on $\cX$, then for a general
point $s\in S(k)$, $\widetilde{f}|_{\pi^{-1}(s)}$ is a stacky curve
which maps the generic point into $\cT$. Conversely, for a stacky
curve $f\colon\cC\to\cX$ which maps the generic point into $\cT$,
we consider the composite morphism
\[
\widetilde{f}\colon T\times\cC\to T\times\cX\to\cT\times\cX\to\cX.
\]
Here $T\times\cX\to\cT\times\cX$ is the product of the standard atlas
$T\to\cT$ and the identity morphism $\id_{\cX}$, and $\cT\times\cX\to\cX$
is the action of $\cT$ on $\cX$. The pair of the projection $\pi\colon T\times\cC\to T$
and $\widetilde{f}$ is a covering family. From Proposition \ref{prop:orb-pointwise}
and Corollary \ref{cor:orb-same}, for every $t\in T(k)$, $[\widetilde{f}|_{\pi^{-1}(t)}]_{\orb}=[f]_{\orb}$,
which completes the proof.
\end{proof}
\begin{thm}
\label{thm:main}The cone $\Mov_{1,\orb}(\cX)$ is generated by the
image $\overline{\Xi_{\eta}}$ of $\Xi_{\eta}$, $\eta\in\Sigma(1)\cup\pi_{0}^{*}(\cJ_{0}\cX)$:
\[
\Mov_{1,\orb}(\cX)=\N_{1,\orb}(\cX)\cap\sum_{\eta\in\Sigma(1)\cup\pi_{0}^{*}(\cJ_{0}\cX)}\RR_{\ge0}\Xi_{\rho}.
\]
\end{thm}

\begin{proof}
Our proof is a variation of the proof of \cite[Proposition 2]{payne2006stablebase}.
From Corollary \ref{cor:orb-class}, the left hand side is included
in the right hand side. To show the opposite inclusion, we take an
arbitrary element of 
\[
\sum_{i=1}^{r}\Xi_{\eta_{i}}\in\sum_{\eta\in\Sigma(1)\cup\pi_{0}^{*}(\cJ_{0}\cX)}\ZZ_{\ge0}\Xi_{\eta}\cap\N_{1,\orb}(\cX),
\]
where $\eta_{1},\dots,\eta_{r}$ is a sequence of elements of $\Sigma(1)\cup\pi_{0}^{*}(\cJ_{0}\cX)$.
We will show that $\sum_{i=1}^{r}\Xi_{\eta_{i}}$ is an element of
$\Mov_{1,\orb}(\cX)$.

Let $\lambda_{1},\dots,\lambda_{r}\in k^{*}$ be distinct elements
and consider the morphism
\[
\Phi(z)=\prod_{i}\phi_{b_{\eta_{i}}}(z-\lambda_{i})\colon\GG_{m}\setminus\{\lambda_{1},\dots,\lambda_{r}\}\to\cT.
\]
Here the product is taken with respect to the group structure of $\cT$.
Let $f\colon\cC\to\cX$ be the stacky curve uniquely extending $\Phi$.
We identify $\cC\langle k\rangle$ with $\PP^{1}\langle k\rangle=k\cup\{\infty\}$.
\begin{claim}
\label{claim:1}The point $\infty\in\cC\langle k\rangle$ is not a
stacky point. Moreover, the image of $\infty$ by $\Phi$ in $\cT\langle k\rangle=T\langle k\rangle$
is the identity element $1_{T}$ for the group structure of $T\langle k\rangle$.
\end{claim}

\begin{proof}[Proof of Claim \ref{claim:1}]
We first prove the second assertion. From (\ref{eq:Ker-orb}), we
have $\sum b_{\eta_{i}}=0$. Since $\Psi(z)=\prod\phi_{b_{\eta_{i}}}(z)$
is the trivial morphism $1\colon\GG_{m}\to\cY$, we have
\begin{align*}
\Phi(\infty) & =\lim_{z\to\infty}\prod_{i}\phi_{b_{\eta_{i}}}(z-\lambda_{i})\\
 & =\lim_{z\to\infty}\prod_{i}\phi_{b_{\eta_{i}}}(z-\lambda_{i})/\phi_{b_{\eta_{i}}}(z)\\
 & =\lim_{z\to\infty}\prod_{i}\phi_{b_{\eta_{i}}}(1-\lambda_{i}/z)\\
 & =\prod_{i}\phi_{b_{\eta_{i}}}(1)\\
 & =1_{T},
\end{align*}
as desired. Considering the affine line $\AA^{1}$ with coordinate
$w=1/z$, we have the morphism 
\[
\prod_{i}\phi_{b_{\eta_{i}}}(1-\lambda_{i}w)\colon\AA^{1}\setminus\{\lambda_{1}^{-1},\dots,\lambda_{r}^{-1}\}\to\cT,
\]
which corresponds to $\Phi$ via the coordinate change $w=1/z.$ This
shows that the point $z=\infty$ (corresponding to $w=0$) is not
a stacky point of $\cC$. 
\end{proof}
\begin{claim}
\label{claim:2}For $i\in\{1,\dots,r\}$, we have $[\Phi_{\lambda_{i}}]_{\orb}=\Xi_{\eta_{i}}$.
\end{claim}

\begin{proof}[Proof of Claim \ref{claim:2}]
For the parameter $t:=z-\lambda_{i}$, $\phi_{b_{\eta_{i}}}$ induces
the $k\tpars$-point $\widehat{\phi_{b_{\eta_{i}}}}$ of $\cT$ with
the induced twisted arc $\gamma_{i}\colon\cD^{l}\to\cX$ and $\phi_{b_{\eta_{j}}}$
with $j\ne i$ induce $k\tpars$-points of $\cT$ which extend to
arcs $\delta_{j}\colon D\to\cT$. We have 
\[
\Phi_{\lambda_{i}}=\delta_{1}\cdot\cdots\cdot\hat{\delta_{i}}\cdot\cdots\delta_{r}\cdot\gamma_{i}.
\]
From Corollary \ref{cor:Xi}, we have $\left[\gamma_{i}\right]_{\orb}=\Xi_{\eta_{i}}$.
From Corollary \ref{cor:orb-same}, $[\Phi_{\lambda_{i}}]_{\orb}=\Xi_{\eta_{i}}$,
as desired.
\end{proof}
Claims \ref{claim:1} and \ref{claim:2} show 
\[
[\Phi]_{\orb}=\sum_{i=1}^{r}[\Phi_{\lambda_{i}}]_{\orb}=\sum_{i=1}^{r}\Xi_{\eta_{i}}.
\]
We have completed the proof of Theorem \ref{thm:main}.
\end{proof}
\begin{cor}
\label{cor:main}For $\eta\in\Sigma(1)\cup\pi_{0}^{*}(\cJ_{0}\cX)$,
let $\{\Xi_{\eta}\ge0\}$ denote the half-space $\{\theta\in U_{\orb}\mid(\theta,\Xi_{\eta})\ge0\}$
in $U_{\orb}$. Then, the orbifold pseudo-effective cone $\PEff_{\orb}(\cX)$
is the image of the cone 
\[
\bigcap_{\eta\in\Sigma(1)\cup\pi_{0}^{*}(\cJ_{0}\cX)}\{\Xi_{\eta}\ge0\}\subset V_{\orb}
\]
by the map $\lambda_{\orb}\colon V_{\orb}\to\N_{\orb}^{1}(\cX)$.
Moreover, it is generated by the $\#(\Sigma(1)\cup\pi_{1}^{*}(\cJ_{0}\cX))$
elements
\[
[\cE_{\rho}]-\sum_{\cY\in\pi_{0}^{*}(\cJ_{0}\cY)}a_{\rho}(\cY)u_{\cY}\quad(\rho\in\Sigma(1))
\]
and 
\[
u_{\cY}\quad(\cY\in\pi_{0}^{*}(\cJ_{0}\cX)).
\]
\end{cor}

\begin{proof}
Recall that the cone $\PEff_{\orb}(\cX)$ is the dual cone of $\Mov_{1,\orb}(\cX)$.
If we regard $\Mov_{1,\orb}(\cX)$ as a cone in $V_{\orb}$, then
it is the intersection of the two cones
\[
\N_{1,\orb}(\cX)\text{ and }\sum_{\eta\in\Sigma(1)\cup\pi_{0}^{*}(\cJ_{0}\cX)}\RR_{\ge0}\Xi_{\rho}.
\]
Thus, its dual cone in $U_{\orb}$ is 
\[
\N_{1,\orb}(\cX)^{\vee}+\left(\sum_{\eta\in\Sigma(1)\cup\pi_{0}^{*}(\cJ_{0}\cX)}\RR_{\ge0}\Xi_{\rho}\right)^{\vee}.
\]
This shows that
\begin{align*}
\PEff_{\orb}(\cX) & =\lambda_{\orb}\left(\N_{1,\orb}(\cX)^{\vee}+\left(\sum_{\eta\in\Sigma(1)\cup\pi_{0}^{*}(\cJ_{0}\cX)}\RR_{\ge0}\Xi_{\rho}\right)^{\vee}\right)\\
 & =\lambda_{\orb}\left(\left(\sum_{\eta\in\Sigma(1)\cup\pi_{0}^{*}(\cJ_{0}\cX)}\RR_{\ge0}\Xi_{\rho}\right)^{\vee}\right)\\
 & \lambda_{\orb}\left(\bigcap_{\eta\in\Sigma(1)\cup\pi_{0}^{*}(\cJ_{0}\cX)}\{\Xi_{\eta}\ge0\}\right).
\end{align*}
We have proved the first assertion.

To show the second assertion, we define $\Xi_{\eta}^{*}\in U_{\orb}$,
$\eta\in\Sigma(1)\cup\pi_{0}^{*}(\cJ_{0}\cX)$ as follows:
\[
\Xi_{\eta}^{*}:=\begin{cases}
c_{\rho}^{-1}\bu_{\rho}-\sum_{\cY\in\pi_{0}^{*}(\cJ_{0}\cY)}a_{\rho}(\cY)\bu_{\cY} & (\eta=\rho\in\Sigma(1))\\
\bu_{\cY} & (\eta=\cY\in\pi_{0}^{*}(\cJ_{0}\cX)).
\end{cases}
\]
We see that $(\Xi_{\eta})_{\eta\in\Sigma(1)\cup\pi_{0}^{*}(\cJ_{0}\cX)}$
is a basis of $V_{\orb}$ and $(\Xi_{\eta}^{*})_{\eta\in\Sigma(1)\cup\pi_{0}^{*}(\cJ_{0}\cX)}$
is its dual basis. It follows that $\sum_{\eta}\RR_{\ge0}\Xi_{\eta}$
is a full-dimensional simplicial cone in $V_{\orb}$ and that its
dual cone $\bigcap_{\eta\in\Sigma(1)\cup\pi_{0}^{*}(\cJ_{0}\cX)}\{\Xi_{\eta}\ge0\}$
is identical to $\sum_{\eta}\RR_{\ge0}\Xi_{\eta}^{*}$. Since $\lambda_{\orb}(\bu_{\cY})=u_{\cY}$
and $\lambda_{\orb}(\bu_{\rho})=[E_{\rho}]=c_{\rho}[\cE_{\rho}]$,
we have
\[
\lambda_{\orb}(\Xi_{\eta}^{*})=\begin{cases}
[\cE_{\rho}]-\sum_{\cY\in\pi_{0}^{*}(\cJ_{0}\cY)}a_{\rho}(\cY)u_{\cY} & (\eta=\rho\in\Sigma(1))\\
u_{\cY} & (\eta=\cY\in\pi_{0}^{*}(\cJ_{0}\cX)).
\end{cases}
\]
This shows the second assertion.
\end{proof}
\bibliographystyle{alphaurl}
\bibliography{Eff-Toric}

\end{document}